\documentclass{article}

\usepackage{natbib}
\setcitestyle{authoryear,open={(},close={)}} 
\usepackage{float}
\usepackage{xcolor}
\usepackage{graphicx}
\usepackage{hyperref}
\usepackage{amsmath, amsthm, amssymb, amsfonts}
 \usepackage{amsfonts}
 \usepackage{amsmath}

 \usepackage{amsthm}
 \usepackage{graphicx}
 \usepackage{hyperref}
 \usepackage{bigints}
 \usepackage{wrapfig} 
 \usepackage{caption}
\usepackage{subcaption}
\usepackage{mathtools}

\usepackage{changes}

\definechangesauthor[name=Soe, color=blue]{so}
\definechangesauthor[name=Ern, color=violet]{em}

\usepackage[english]{babel}
\newtheorem{thm}{Theorem}[section] 
\newtheorem{theorem}[thm]{Theorem}
\newtheorem{assumptions}[thm]{Assumption}

\newtheorem{corollary}[thm]{Corollary}
\newtheorem{lemma}[thm]{Lemma}
\newtheorem{proposition}[thm]{Proposition}
\newtheorem{remark}[thm]{Remark}
\newtheorem{definition}[thm]{Definition}

\newcommand{\R}{\mathbb{R}}

\renewcommand\P{\operatorname{\mathbf{P}}}
\newcommand\E{\operatorname{\mathbf{E}}}
\makeatletter
\newcommand*\bigcdot{\mathpalette\bigcdot@{.5}}
\newcommand*\bigcdot@[2]{\mathbin{\vcenter{\hbox{\scalebox{#2}{$\m@th#1\bullet$}}}}}
\makeatother

\title{Mean-Field Games with two-sided singular controls for L\'evy processes}
\author{Facundo Oli\'u
\thanks{Ingeniería Forestal, Centro Universitario de Tacuarembó, Universidad de la República,
            Road 5, km. 386.500, 
            Tacuarembó city,
            Tacuarembó,
            Uruguay}}

\begin{document}
\maketitle

\begin{abstract}

In a probabilistic mean-field game driven by a L\'evy process an individual player aims to minimize a long-run discounted/ergodic cost by controlling the process through a pair of –increasing and decreasing– c\`adl\`ag processes, while he is interacting with an aggregate of players through the expectation of a controlled process by another pair of c\`adl\`ag processes.
With the Brouwer fixed point theorem, we provide easy to check conditions for the existence of mean-field game equilibrium controls for both the discounted and ergodic control problem, characterize them as the solution of an integro-differential equation and show with a counterexample that uniqueness does not always holds. Furthermore, we study the convergence of equilibrium controls in the abelian sense.
Finally, we treat the convergence of a finite-player game to this problem to justify our approach.
\end{abstract}

\section{Introduction}\label{Ch4:Introduction}
Mean-field game theory has emerged for modeling the behavior of large populations of interacting players in a stochastic environment.
 In a \emph{strategic game} or $N$-player game,
in the search of \emph{Nash equilibrium}, when a large number of players try to minimize their respective cost functions, the natural simplification is to treat \emph{asymptotic problem}. This new problem where the players are replaced by a measure, or in some cases a function, is called \emph{mean-field game} or for short MFG.
 As a reference we cite \cite{CD(2018)} \& \cite{LL}. For related MFG problems see \cite{CFL}, \cite{Lacker}, \cite{FU},  \cite{LackerZ}, \cite{GR}, \cite{Tobi},  \cite{CG}, \cite{ABF}, \cite{CDF}, \cite{BDTY}, \cite{BCP}. We specifically want to mention the \emph{impulse MFG control problem}  \cite{Tobi} and the singular MFG control problem for diffusions \cite{CMO} as their MFG formulation are similar to ours.
 
As applications, for instance, in finance and energy systems, see \cite{Carmona}, in traffic management and social dynamics, see \cite{ADRIANOFIESTA!}, and in machine learning, see \cite{SUBRAMANIAN}. 

The problems in which we add an MFG component fall under the class of \emph{bounded variation, stochastic singular control optimization discounted and ergodic problems}. As we are working in the real line, we can make a distinction between the cost of the increasing part of the control and the cost of the decreasing one. In stochastic bounded variation control, the displacement of the state caused by the control is of bounded variation. Moreover, we say that it is singular when the optimal control is the reflection inside an interval. The reason for this name is that for the Brownian motion (not necessarily for every L\'evy processes) the trajectory of the reflected controls in intervals are singular with respect to the Lebesgue measure. 
As references we cite \cite{KARATZAS}, \cite{MenaldiRobin}, \cite{AW}, \cite{AJackZervos}, \cite{AW2},   \cite{MenaldiRobin2013}, \cite{wuchen}, \cite{ACPZ}, \cite{KXYZ}, \cite{Alvarez}. 
\\
With respect to applications of singular control results, we mention studies focusing on cash flow management that investigate optimal dividend distribution, recapitalization, or a combination of both, while considering risk neutrality. 
See, for example, \cite{AT}, \cite{HT}, \cite{JS}, \cite{P}, \cite{BK}, and \cite{SLG}. 
Moreover, the ergodic model comes pretty handy  in problems of finding sustainable harvesting policies, see \cite[Chapter I]{CLARK}.
\\
Our departure points are the singular control problems posed in \cite{MO}, where the underlying processes are L\'evy. In that framework, in the present paper, we incorporate a mean-field game
dependence into the two-sided discounted and ergodic singular control problem. As a consequence, we obtain sufficient conditions
for the existence of mean-field game equilibrium points and characterize them as the solution of an integro-differential equation. \\
Finally, we define an $N$-player problem and prove that a mean-field equilibrium
is an approximate Nash equilibrium for the $N$-player game.
 \\
The rest of the paper is organized as follows.
In Section \ref{C43:Setting} we define the framework and provide the main results of this article.
In Section \ref{C43:FixedpointTheorem} we use the adjoint Dynkin game to prove that there is a MFG equilibrium for the $\epsilon$-discounted control problem, for that endeavor we use Brouwer fixed point Theorem.
In Section \ref{C43:MFGequilibriumforergodic} we use regenerative theory to prove that the equilibrium points in the discounted case have a convergent subsequence to a MFG equilibrium for the ergodic problem.
In Section \ref{C43:Examples} we provide examples and in Remark \ref{Remark:Nonuniqueness} we show that uniqueness does not needs to hold. Finally in Section
\ref{C4:Nplayer} we study the convergence of the $N$-player game to the MFG for both problems.


\section{Framework and main results}\label{C43:Setting}

\subsection{L\'evy processes, controls and cost functions}\label{S:framework}
Let $X=\{X_t\}_{t\geq 0}$ be a L\'evy process with finite mean that is neither a subordinator, nor the opposite of a subordinator on a stochastic basis
${\cal B}=(\Omega, {\cal F}, {\bf F}=({\cal F}_t)_{t\geq 0}, \mathbf{P}_x)$ departing from $X_0=x$.
Assume that the filtration is right-continuous and complete (see \cite[Definition 1.3]{JJAS}).
Denote by $\E_x$ the expected value associated to the probability measure $\mathbf{P}_x$, let $\E= \E_0$ and $\mathbf{P}=\mathbf{P}_0$.   
The L\'evy-Khintchine formula characterizes the law of the process,
stating
$$
\phi (z)= \log \left( \E( e^{z X_1}) \right), \qquad z=i\theta \in i\mathbb{R},
$$
with
\begin{equation*}
\phi(z)={\sigma^2\over 2}z^2+z \mu+\int_{\R}\left(e^{z y}-1-z y\right)\Pi(dy),
\end{equation*}
where $\mu=\E(X_1)\in\R$, $\sigma\geq 0$ and $\Pi(dy)$ is a non-negative measure (the \emph{jump measure}) 
that satisfies in our case $\int_{\R}(y^2\wedge |y|)\Pi(dy)<\infty$. 
This L\'evy process, being a special semimartingale (see \cite[Chapter II, 2.29]{JJAS}), 
can be expressed as a sum of three independent processes 
\begin{equation}\label{D:LEVYPROCESSEQUATION}
X_t= X_0 +\mu t +\sigma W_t + \int_{[0,t]\times\R} y\, \tilde N(ds,dy),
\end{equation}
where $\tilde N(ds,dy)=N(ds,dy)-ds\Pi(dy)$ is a compensated Poisson random measure, 
$N(ds,dy)$ being the jump measure constructed from $X$ 
and $\lbrace W_t \rbrace_{t \geq 0}$ is an independent Brownian motion. \\
In the case when $X$ has bounded variation, we denote $\lbrace S_t^+ \rbrace_{t \geq 0}, \ \lbrace S_t^- \rbrace_{t \geq 0}$ the couple of independent subordinantors starting at zero such that for all $t \geq 0$
$$ X_t=x +S_t^+ -S_t^-.$$

For general references on L\'evy processes see \cite{B,KRI,KIS}. We proceed to define the set of \emph{admissible controls}, in this case there is no stochastic differential equation due to the fact that L\'evy processes have stationary increments. 

\begin{definition}\label{D:AdmissiblecontrolsLEVY} 
An \emph{admissible control} is a pair of non-negative $\mathbf{F}$-adapted processes $(U,D)$ such that:
\vskip1mm\par\noindent
{\rm(i)} Each process $U,D\colon\Omega \times \mathbb{R}_+ \rightarrow \mathbb{R}_+$ is right continuous and nondecreasing almost surely.
\vskip1mm\par\noindent
{\rm(ii)} For each $t\geq 0$ the random variables $U_t$ and $D_t$ have finite expectation.
\vskip1mm\par\noindent
We denote by $\mathcal{A}$ the set of admissible control.
\end{definition}
A controlled L\'evy process by the pair $(U,D) \in \mathcal{A}$ is defined as
\begin{equation}\label{D:controlledequation}
X^{U,D}_t= X_t +U_t-D_t, \qquad X_0=x,\ U_{0}=u_0,  \ D_{0}=d_0.
\end{equation}
For $a<b$ let $\lbrace X^{a,b}_t\colon \ t \geq 0 \rbrace$ be a process defined on $(
\Omega,\mathcal{F}, {\bf F},\mathbf{P
}_x)$ that follows \eqref{D:controlledequation} where 
$ U^{a,b}_t,-D^{a,b}_t$ are the respective reflections at $a$ and $b$,  called \emph{reflecting controls}.
At time zero, the reflecting controls are defined as $U^{a,b}_0=(a-x)^+$ and $D^{a,b}_0=(x-b)^+$ and from now on are denoted $u_0^{a,b}$ and $d_0^{a,b}$ respectively. Moreover $ U^{a,b}, \ D^{a,b}$ satisfy

\begin{equation}\label{D:skorhod}
     \int_0^{\infty} (X^{a,b}_t -a)dU^{a,b}_t=0, \quad \int_0^{\infty} (b- X^{a,b}_t )dD^{a,b}_t=0 . 
     \end{equation}
There is an unique strong solution that satisfies \eqref{D:controlledequation} and is also a solution of the Skorokhod \eqref{D:skorhod} problem (see \cite{AAGP}). We remark that for every exponential random variable $e(\epsilon)$ with parameter $\epsilon>0$ independent of the process $X$ and every $t>0$ the random variables $U^{a
,b}_{e(\epsilon)} \ U^{a,b}_t, \ D^{a,b}_{e(\epsilon)}, \  D^{a,b}_t, $ have finite mean as a consequence of \cite[Theorem 6.3]{AAGP}. When the process has bounded variation we also define 
$$(U^{0,0}_t,D^{0,0}_t)=(S^-_t,S^+_t ), \text{ for } t>0, \quad (U^{0,0}_0,D^{0,0}_0)= (-\min \lbrace x,0 \rbrace , \max \lbrace x,0 \rbrace) $$
as a reflecting control. Observe that formula \eqref{D:controlledequation} holds with $X_t^{U^{0,0},D^{0,0}}=0$, for all $t \geq 0$.

From now on $q_u, q_d $ are positive constants, 
and we refer to them as lower barrier cost and upper barrier cost, respectively. We also assume that there is a real continuous map $f$  that models the MFG component.  

\begin{definition}\label{Ch4DefinitionStationarymeasure}
We denote $\mathcal{P}^{\infty} $ the set of random variables $X^{\eta}_{\infty}$ with compact support, such that  there exists $\eta=(U,D) \in \mathcal{A}$ so that for every $ t \geq 0$: 
 $$\lim_{t \to \infty}X_t^{U,D}= X^{\eta}_{\infty},   \text{ in distribution} .$$
 For simplicity we denote $p^{\eta}:= \E \left(f(X_{\infty}^{\eta}) \right). $ Moreover when there is no need to highlight the importance of $\eta$ we simply denote  $p^{\eta} $ as $p$.
\end{definition}

\begin{remark}
Let $a<b,x \in \mathbb{R}$, then the probability flux $\mathbf{P}_x(X^{a,b}_t \in dx)$ converges in total variation to a stationary distribution in $\mathcal{P}^{\infty}$ (see \cite[Section V]{AAGP}).
\end{remark}
We define, when $a <b$ the value $p^{a,b}:=p^{(U^{a,b},D^{a,b})}$ and in the degenerate case $a=b$ we define $p^{a,a}:=f(a) $.

\begin{proposition}
 For every $a \in \mathbb{R}$, the constant random variable $X^a:=a $ belongs to $\mathcal{P}^{\infty}$. 
\end{proposition}
\begin{proof}
    The main idea is to take an smaller and smaller reflection at each time interval. It is clear that it is enough to prove the proposition for $a=0$ and that we can assume $x=0$. For $t \in [0,1]$, let  
    $$Y^n_t :=X_{t+n}-X_n . $$
Observe that $\lbrace Y^n_t \rbrace_{n \geq 0}$ is a sequence of independent L\'evy processes starting at zero that satisfy for every $n$, $Y^n$ is independent of $\sigma (X_u : u \leq n) $. Moreover $X_t$ can be rewritten as:
$$X_t=\sum_{n=0}^{[t]} Y^n_{(t-n) \wedge 1} . $$
We proceed to reflect each process $Y^n$ at $-1/n,0$ in the following way:
    \begin{align*}
    & U^n_t, D^n_t \text{ the reflecting controls at } -1/n, \ 0 \text{ of }Y^n_t , \ t \in [0,1], \ n \geq 0  \\
    &d^0=0, \quad  d^n=(Y^{n-1}_1+U^{n-1}_1-D^{n-1}_1)^+, \ n \geq 1, \\\
    &u^0=0, \quad  u^n=-(Y^{n-1}_1+U^{n-1}_1-D^{n-1}_1)^-, \ n \geq 1, \\\
    \end{align*}
Define the admissible controls $U_t,D_t$ as:
$$U_t:= \sum_{n=0}^{[t]}\left( U^n_{(t-n) \wedge 1 }+u^n \right) , \  D_t:= \sum_{n=0}^{[t]}\left( D^n_{(t-n) \wedge 1 } +d^n \right)  .$$
It is clear $(U_t,D_t)$ is an increasing and adapted process and it is satisfied (except in the set that $X_t$ has a jump in a natural number) that $X_t +U_t-D_t \in [-1/n,0]$ for $t \geq n$. Then we deduce, $\lim_{t \to \infty}X_t+
U_t-D_t=0$ a.s, thus concluding the proof. 
\end{proof}

We define the running cost, now depending on $X^{\eta}_{\infty} \in \mathcal{P}^{\infty} $ too, in a way that for every fixed $\eta$ the map $c(\cdot, \E f\left(X^{\eta}_{\infty} \right)) $ defines a singular control problem as the one in \cite{MO}.

\begin{assumptions}\label{A:Assumptions2MEANFIELDGAME}

We assume that the function
$c \colon \mathbb{R}^2  \rightarrow \mathbb{R}_+$ is continuous, non-negative and

\vskip1mm\par\noindent
{\rm(i)} 
for every $y \in \mathbb{R}$, the function $c(\cdot,y)$ is strictly convex, has a global minimum in the first variable which is reached at zero and for every $r>0$

$$\inf_{(x,y) \in  (-r,r)^c \times \mathbb{R} }c_{xx}(x,y)>0. $$

\vskip1mm\par\noindent
{\rm(ii)}  For each fixed $y \in \mathbb{R}$, there is a pair of positive constants $N$ (independent of $y$) and $M_y$ that satisfy
$$ c(x,y)+ M_y \geq N \vert x  \vert  , \  \forall x \in \mathbb{R} .$$

\vskip1mm\par\noindent
{\rm(iii)} 
For every $(x,y) \in \mathbb{R}^2$
 \[\E_x \bigg(\int_0^{\infty}c_x(X_s,y)e^{-\epsilon s}ds \bigg) < \infty , \ \forall x \in \mathbb{R} .\] 

\begin{remark}
To verify {\rm(iii)} , if for each fixed $y \in \mathbb{R}$, there is a real function $f^{y}$ such that 
$\vert c_{x}(x,y) \vert =f^{y}(\vert x \vert)$ and a  constant $K^{y}>0$ such that 
$f^{y}(x_1+x_2) \leq K^{y} f^{y}(x_1) f^{y}(x_2)$, 
according to \cite{KIS}, Theorem 25.3, Lemma 25.5 and Theorem 30.10,
it is enough to check
$$
\int_{\vert x \vert \geq 1} f^{y}( \vert x\vert ) \Pi(dx) <\infty.
$$
\end{remark}

\begin{definition}\label{D:ergodicvaluefunctionMFG}
Given $x\in\R, \ X^{\eta}_{\infty} \in \mathcal{P}^{\infty}$ and a control $(\hat{U},\hat{D})\in\mathcal{A}$, we define the ergodic cost function

\begin{equation*}
J(x,(\hat{U},\hat{D}), X^{\eta}_{\infty} )
=  \limsup_{T \to \infty} \frac{1}{T} \E_x  \left(   \int_{0}^T c(X^{\hat{U},\hat{D}}_s,p^{\eta})ds    +q_u \hat{U}_T +q_d \hat{D}_T \right),
\end{equation*}
and the ergodic value function 
\[
G(x, X^{\eta}_{\infty} ) = \inf_{(\hat{U},\hat{D}) \in \mathcal{A}} J(x,(\hat{U},\hat{D}), X^{\eta}_{\infty} ).
\]
\end{definition}

\begin{definition}\label{D:discountedgameMFG}
Given $x\in\R, \  X^{\eta}_{\infty}  \in \mathcal{P}^{\infty},$ a control $(\hat{U},\hat{D})\in\mathcal{A}$ and a fixed positive constant  $\epsilon$,
we define the $\epsilon$-discounted cost function 
\begin{multline*}
J_{\epsilon}(x,(\hat{U},\hat{D}), X^{\eta}_{\infty} )  \\ =   \E_x\left( \int_{0}^{\infty}e^{-\epsilon s}(c(X_s^{\hat{U},\hat{D}},p^{\eta})ds + q_ud\hat{U}_s + q_d  d\hat{D}_s\right) +\hat{u}_0 q_u+\hat{d}_0 q_d,  
\end{multline*}
and the $\epsilon$-discounted value function
\[
G_{\epsilon}(x, X^{\eta}_{\infty} )= \inf_{(\hat{U},\hat{D}) \in \mathcal{A}} J_{\epsilon}(x,(\hat{U},\hat{D}), X^{\eta}_{\infty} ).
\]

\end{definition}

\begin{definition}\label{D:EquilibriumMFG}
We say that a  pair of points $(a,b)$ with $a \leq b$ (the inequality strict if $X$ has unbounded variation) is :
\vskip1mm\par\noindent
{\rm(i)} 
an $\epsilon$-discounted equilibrium if $G_{\epsilon}(x, X^{a,b}_{\infty} )=J_{\epsilon}(x,U^{a,b},D^{a,b}, X^{a,b}_{\infty} ) $ for all $x \in \mathbb{R}$,
\vskip1mm\par\noindent
{\rm(ii)}
an ergodic equilibrium if $G(x, X^{a,b}_{\infty} )=J(x,U^{a,b},D^{a,b}, X^{a,b}_{\infty} )$ for all $x \in \mathbb{R}$.
\end{definition}

\subsection{Main results}
 The most important results of this article are the existence of an equilibrium the discounted problem  and the existence of a sequence of $\epsilon$-discounted equilibrium points converging to an ergodic equilibrium point. These results are redacted in the following two Theorems which are proven in the sections
 \ref{C43:FixedpointTheorem} and \ref{C43:MFGequilibriumforergodic}. For $a, b \in \mathbb{R}$, let us define  $\tau(a)$ and $\sigma(b)$ as:
 $$
    \tau(a)=\inf\{t\geq 0\colon X_t\leq a\},
\quad
\sigma(b)=\inf\{t\geq 0\colon X_t\geq b\}. $$

\begin{theorem}\label{T:MEANFIELDDISCOUNTEDPROBLEMSOLUTION}
Under the  hypotheses posed in this section:

\vskip1mm\par\noindent
{\rm(i)} 
 for all $X^{\eta}_{\infty} \in \mathcal{P}^{\infty}, \  G_{\epsilon}(\cdot,X^{\eta}_{\infty})$ is in the domain in the infinitesimal generator,

\vskip1mm\par\noindent
{\rm(ii)}  a pair $(a,b)$ is an $\epsilon$-discounted equilibrium if $u(x):=G_{\epsilon}(x,X^{a,b}_{\infty})$  satisfies that the map
$$ x  \rightarrow \mathcal{L}u(x)- \epsilon u(x)+c(x,p^{a,b}), $$
is equal or smaller than zero, constant zero and increasing in the sets $(-\infty, a] , \ (a,b) $ and $[b,\infty)$ respectively

\vskip1mm\par\noindent
{\rm(iii)} 
There is an $\epsilon$-discounted equilibrium  
 $(a^{\ast}_{\epsilon}, b^{\ast}_{\epsilon})$, with $a^*_{\epsilon} < 0 < b^*_{\epsilon} $, that satisfies $\rm{(ii)}$ and
\begin{align*}
    V(x,X^{a^*_{\epsilon},b^*_{\epsilon}}_{\infty})  = 
     \sup_{a \leq 0} \inf_{b \geq 0}\E_x & \left(\int_0^{\tau(a) \wedge \sigma(b)} c_x(X_s,p^{a^{\ast}_{\epsilon},b^{\ast}_{\epsilon}})e^{-\epsilon s}ds \right. \\ 
     &  +q_d e^{-\epsilon \sigma(b)} \mathbf{1}_{\tau(a) \geq \sigma(b) }  
     -q_ue^{-\epsilon \tau(a)} \mathbf{1}_{\sigma(b) >\tau(a)}  \Bigg)  \\
     =\E_x &   \left(\int_0^{\tau(a^{\ast}_{\epsilon}) \wedge \sigma(b^{\ast}_{\epsilon})}  c_x(X_s,p^{a^{\ast}_{\epsilon},b^{\ast}_{\epsilon}})  e^{-\epsilon s}ds \right.  \\
     &  +q_d e^{-\epsilon \sigma(b^{\ast}_{\epsilon})} \mathbf{1}_{\tau(a^{\ast}_{\epsilon}) \geq \sigma(b^{\ast}_{\epsilon}) }  
     -q_ue^{-\epsilon \tau(a^{\ast}_{\epsilon})} \mathbf{1}_{\sigma(b^{\ast}_{\epsilon}) >\tau(a^{\ast}_{\epsilon})}   \Bigg) .
    \end{align*}

Furthermore $ a^{\ast}_{\epsilon},b^{\ast}_{\epsilon}$   satisfy:
\begin{equation*}
a^{\ast}_{\epsilon}=\sup \{x, V(x,X^{a^*_{\epsilon},b^*_{\epsilon}}_{\infty})=  -q_u \},\ 
b^{\ast}_{\epsilon}=\inf \{x, V(x,X^{a^*_{\epsilon},b^*_{\epsilon}}_{\infty})= q_d \}.    
\end{equation*}
\end{theorem}
\begin{remark}
    The statement \rm{(i)} is obtained from \cite[Theorem 1, and propositions 7 \& 8]{MO}. In fact it could have been omitted but we considered it to be informative.
\end{remark}
\begin{theorem}\label{T:ErgodicProblemsolutionMFG}
There is an ergodic equilibrium $(a^{\ast},b^{\ast})$, with $a^*\leq 0 \leq b^*$, that satisfies:
\vskip1mm\par\noindent
{\rm(i)}
There is a sequence $\lbrace (a_{\epsilon_n}^{\ast},b_{\epsilon_n}^{\ast} , \epsilon_n) \rbrace_{n \geq 0} $ converging to $ (a^{\ast},b^{\ast},0)$ when $n \to \infty$, such that $(a_{\epsilon_n}^{\ast},b^{\ast}_{\epsilon_n})$ is an $\epsilon_n$-discounted equilibrium and
\vskip1mm\par\noindent
{\rm(ii)} for every $x\in \mathbb{R}$  
$$\lim_{n \to \infty} \epsilon_n G_{\epsilon_n}(x, X^{a^*_{\epsilon_n},b^*_{\epsilon_n}}_{\infty}) = G(x,X^{a^,b^*}_{\infty}).$$
\end{theorem}

\end{assumptions}

\section{Mean-field equilibrium of the discounted problem}\label{C43:FixedpointTheorem}
Due to the fact that for every $y \in \mathbb{R}$, the map $x \rightarrow c(x,y)$ is twice countinuously differentiable and with positive bounded below second derivate outside intervals containing zero in the first coordinate, useful properties of the associated Dynkin game (see \cite{MO}) are deduced in the first part of this section. Then, in the second part, we use these properties to show that there is an MFG equilibrium with the Brouwer fixed point Theorem. 

\subsection{The adjoint Dynkin game}
The relationship between singular control problems and Dynkin games has been known for a certain time, see \cite{KARATZAS, KW,boetius, GT,GT0} and more specifically \cite{MO}, for a fixed $y \in \mathbb{R}$, the results of that article can be applied when the cost function is the map $x \rightarrow c(x,y)$. We borrow the notations from that paper with the difference that we add the suffix $p$ in every definition to highlight the dependence of $p= \E \left(f(X^{\eta}_{\infty}) \right)$. To be more specific, for $x \in \mathbb{R}$, $\tau, \ \sigma$ stopping times:
\begin{itemize} 
    \item[\rm{i)}] The constant $\epsilon>0$ and the parameter $p^{\eta}$ is fixed and denoted $p$ for simplicity.
    \item[\rm{ii)}] We work with the same filtered probability space and we define the auxiliary functions:
    \begin{align*}
          &      Q_r(\tau,\sigma)=e^{-r\epsilon}\left(
        q_d\mathbf{1}_{\{\tau \geq \sigma\}}e^{-\epsilon\sigma}  
        -q_u\mathbf{1}_{\{\tau < \sigma\}}e^{-\epsilon\tau} 
        \right),  \\
        & M^p_x(\tau,\sigma)=
\E_x\left(\int^{\tau\wedge\sigma}_0 c_x(X_s,p)e^{-\epsilon s}ds+Q_r(\tau,\sigma)\right),   \\
& V_1(x)  = \sup_{\tau} \inf_{\sigma} M_x(\tau, \sigma)=
V_2(x)= \inf_{\sigma} \sup_{\tau} M_x(\tau, \sigma)=:V_{\epsilon,p}(x),
        \end{align*}
where the last equality is proven in \cite[Proposition 3]{MO}.
\item[\rm{iii)}] The function $V_{\epsilon,p}$ is continuously Lipschitz (\cite[Proposition 7]{MO} and there is a pair $ a^*_{\epsilon,p}<0<b^*_{\epsilon,p}$, called $p,\epsilon$-Nash equilibrium points, such that
$$V_{\epsilon,p}(x)= M_x (\tau (a^*_{\epsilon,p}), \sigma(b^*_{\epsilon,p})  ) \ \text{ for every }  x\in \mathbb{R},$$
with 
$$
a^*_{\epsilon,p}=\sup\{x\colon V_{\epsilon,p}(x)=-q_u\},\quad
b^*_{\epsilon,p}=\inf\{x\colon V_{\epsilon,p}(x)=q_d\},
$$
see \cite[Proposition 5]{MO}. Furthermore for every pair of stopping times $(\tau,\sigma)$:
$$  M_x (\tau , \sigma(b^*_{\epsilon,p})  ) \leq M_x (\tau (a^*_{\epsilon,p}), \sigma(b^*_{\epsilon,p})  ) \leq  M_x (\tau (a^*_{\epsilon,p}), \sigma  ). $$
\end{itemize}

 In \cite[Proposition 6]{MO}, it was proven that there was an interval $[-L,L]$ such that $(a^*_{\overline{\epsilon}},b^{*}_{\overline{\epsilon}}) \in [-L,L]$ for every $\overline{\epsilon} \leq \epsilon$. If we were to use this proposition we should highlight the dependence of $L$ with respect to $p$. However we can refine the proposition so that the bound $L$ does not depends on $p$.
 For that endeavor, for $\ell>0$:
\begin{equation*}
\gamma^{\ell}=\inf \left\lbrace t\geq 0\colon  \left\vert X_t-\ell \right\vert \geq \frac{\ell}{2} \right\rbrace.     
\end{equation*}

\begin{proposition}\label{P:BOUNDSABWITHMEANFIELD}
  For $\epsilon$ small enough there is a constant $L$ such that the optimal thresholds of the $\overline{\epsilon} $-Dynkin game  $(a^{\ast}_{\overline{\epsilon},p },b^{\ast }_{\overline{\epsilon},p }) \subset[-L,L]$  for every $\overline{\epsilon}  \leq \epsilon$ and every $X^{\eta}_{\infty} \in \mathcal{P}^{\infty}$. %
\end{proposition}
\begin{proof}
From Assumptions \ref{A:Assumptions2MEANFIELDGAME}, for $\ell>2$ satisfying 
$$ \ell\left( \inf_{(x,y)\in (-1/2,1/2)^c \times \mathbb{R}  }c_{xx}(x,y) \right) >4N,  $$
we have, for all $x\geq \ell/2$ that,
$
c_x(x,p)= \int_0^x c_{xx}(u,p) du >  N.
$
Then for all $\overline{\epsilon} >0$:
\begin{equation}\label{E:boundDynkin0}
\E_{\ell} \left(\int_0^{\gamma^\ell}c_x(X_s,p)e^{-\overline{\epsilon}  s}ds \right)
\geq N\E_{\ell} \left( \frac{1-e^{-\overline{\epsilon}  \gamma^{\ell}}}{\overline{\epsilon} } \right).
\end{equation}
On the other hand, as
$$
\E_{\ell}(\gamma^\ell)=\E\left(\inf\left\lbrace t\geq 0\colon |X_t|\geq \frac{\ell}{2} \right\rbrace \right)\to\infty\text{ as $\ell\to\infty$},
$$
we find an $\ell$ s.t.
\begin{equation}\label{E:boundDynkin1MEANFIELD}
\E_{\ell} (\gamma^{\ell})  >\frac{1}{N}(q_u+q_d).
\end{equation}
For a fixed $\ell$ satisfying \eqref{E:boundDynkin1MEANFIELD}, using dominated convergence and the fact $\E_{\ell}(\gamma^{\ell})< \infty$ when $X$ is not the null process, we have that 
$$ \E_{\ell} \left( \frac{1-e^{-\overline{\epsilon}  \gamma^{\ell}}}{\overline{\epsilon} } \right) \nearrow  \E_{ \ell}(\gamma^{\ell}) \ \text{ as} \ \overline{\epsilon}  \to 0 .$$
Thus, using \eqref{E:boundDynkin1MEANFIELD} we can take $\epsilon$ small enough such that for every $\overline{\epsilon}  \leq \epsilon$ and every $p \in \mathcal{P}^{\infty}$ we have
\begin{equation}\label{E:boundDynkin0MEANFIELD}
\E_{\ell} \left(\int_0^{\gamma^\ell}c_x(X_s,p)e^{-\overline{\epsilon}  s}ds \right)
\geq q_d+q_u.
\end{equation}

We now take $L:= 3\ell/2$ and prove that for every $X_{\infty}^{\eta} \in \mathcal{P}^{\infty}, \ \overline{\epsilon}  \leq \epsilon, \ b^{\ast}_{\overline{\epsilon},p^{\eta} }\leq L$ (for simplicity let $p:=p^\eta$).
Assume, by contradiction, that $b^{\ast}_{\overline{\epsilon},p }>L$, what implies $q_d>V_p(\ell)$. 
Now we have 
\begin{multline*}
q_d > V_p(\ell)
=\E_{\ell} \left(  \int_0^{\gamma^\ell}c_x(X_s,p)e^{-\overline{\epsilon}  s}ds   
+e^{-\overline{\epsilon}  \gamma_\ell}V_p(X_{\gamma_\ell})\right)
\geq \\  \E_{\ell} \left(   \int_0^{\gamma^\ell}  c_x(X_s,p)e^{-\overline{\epsilon}  s}ds \right) -q_u \geq q_d.
\end{multline*}
by \eqref{E:boundDynkin0MEANFIELD}, what is a contradiction. 
The other bound is analogous, thus we obtain an $\overline{\ell}>0$ such that for every $\overline{\epsilon} \leq \epsilon$ we have $a^{\ast}_{\overline{\epsilon},}>-3 \overline{\ell}/2$. By taking $L= (3/2)\max(\overline{\ell},\ell) $, we conclude the proof.
\end{proof}
For the following statements,  hitting times of translated sets play an important role.
If $\gamma$ is defined by
\[ 
\gamma = \inf \lbrace t\geq 0\colon X_t \in A \rbrace,
\]
then the stopping time $\gamma_y$ is 
\[ 
\gamma_y= \inf \lbrace t\geq 0\colon X_t \in A -y\rbrace.
\]
We proceed to prove that the $p,\epsilon$-Nash equilibrium is unique.
\begin{lemma}\label{L:MEANFIELDDynkinLipschitz}
If there is a couple $(A,B)$ such that the $(\tau^A,\sigma^B)$ defined as
$$\aligned
\tau(A) =\inf\{t\geq 0\colon X_t\leq A\}, \quad
\sigma(B) =\inf\{t\geq 0\colon X_t\geq B\}
\endaligned $$
is a Nash Equilibrium then
\begin{equation}\label{eq:MEANFIELDlimsupderivate}
\liminf_{h \to 0^+}{V_p(x+h)-V_p(x)\over h}\\ \geq  \lim_{h \to 0^+} \E\int_0^{\tau(A)_{x}\wedge \sigma(B)_{x+h}}c_{xx}(x+X_s,p) e^{-\epsilon s}ds.
\end{equation}

\end{lemma}
\begin{proof}
Take $x\in\R$ and $h>0$. We obtain the bound
$$
\aligned
V_p(x+h)  &  \geq  \E_{x+h}\left( \int_0^{\tau(A)_{-h}\wedge\sigma(B)}c_x(X_s,p) e^{-\epsilon s}\ ds   +Q(\tau(A)_{-h},\sigma(B))\right) 
\\
        &=\E\left( \int_0^{\tau(A)_{x}\wedge\sigma(B)_{x+h}}c_x(x+h+X_s,p)e^{-\epsilon s}ds +Q(\tau(A)_x,\sigma(B)_{x+h})\right).\\
\endaligned
$$
Similarly
$$
\aligned
V_p(x)  &\leq  \E_{x}\left(\int_0^{\tau(A) \wedge\sigma(B)_{h}}c_x(X_s,p)e^{-\epsilon s}\,ds+Q(\tau(A),\sigma(B)_h)\right)\\
        &=\E\left( \int_0^{\tau(A)_{x}\wedge\sigma(B)_{x+h}}c_x(x+X_s,p)e^{-\epsilon s}\,ds+Q(\tau(A)_{x},\sigma(B)_{x+h})\right).\\
\endaligned
$$

Subtracting, and applying the mean value theorem, we get
$$
{V_p(x+h)-V_p(x)\over h}\geq \E\left(\int_0^{\tau(A)_{x}\wedge\sigma(B)_{x+h}}c_{xx}(x+X_s+\theta h,p) e^{-\epsilon s}\,ds\right),
$$
where $0\leq\theta\leq 1$. We conclude that the inequality \eqref{eq:MEANFIELDlimsupderivate} holds by taking $\liminf_{h \to 0}$ and observing that the limits
\begin{align*}
 &   \lim_{h \to 0}\E\left(\int_0^{\tau(A)_{x}\wedge\sigma(B)_{x+h}}c_{xx}(x+X_s+\theta h,p) e^{-\epsilon s}\,ds\right), \\
&    \lim_{h \to 0}\E\left(\int_0^{\tau(A)_{x}\wedge\sigma(B)_{x+h}}c_{xx}(x+X_s,p) e^{-\epsilon s}\,ds\right),\end{align*}

exist and are equal due to $c_{xx}$ being continuous and $\sigma(B)_{x+h}$ being monotone in $h$.
\end{proof}

From the previous lemma we deduce that if there was a couple $(A,B) \neq (a^{\ast}_{p,\epsilon},b^{\ast}_{p,\epsilon})$ that was also a Nash equilibrium, we would have different representation for $V_p$ such that the stopping region is different. Thus there would be a point $x$ in one stopping region and in other continuation region. This implies $V_p'(x)>0$ according to one representation and $V_p'(x)=0$ according to the other representation (recall $V_p'$ is Lipschitz thus almost everywhere differentiable). We write this result in the next corollary: 
\begin{corollary}\label{Ch4:UniquenessDynkinMFG}
The only couple of points $(A,B)$ such that $\tau(A) , \sigma(B)$ defined as in the previous Lemma is a Nash Equilibrium is $(a^{\ast}_{p,\epsilon},b^{\ast}_{p,\epsilon})$.
\end{corollary}
\subsection{Fixed point}
To prove that there is an equlibrium in the discounted case, as usual in the literature, we use a fixed point theorem. To be more specific, we use Brouwer's  Fixed Point-Theorem (every continuous function from a compact convex subset of a Euclidean space to itself has a fixed point, see \cite[Section 6]{Brouwer}). We need some properties before defining the adequate function.
In this subsection $L$ is defined as in Proposition \ref{P:BOUNDSABWITHMEANFIELD}. In the following proposition we use implicitly the fact that the exit time of intervals for non-null L\'evy processes have finite mean..

\begin{proposition}\label{P:contintuityinP} For every $r, \epsilon>0, \ X^{\mu_1}_{\infty},X^{\mu_2}_{\infty} \in \mathcal{P}^{\infty}$,  if  $$\sup_{x \in [-L,L]} \vert c'(x, p^{\mu_1}) -c'(x,p^{\mu_2}) \vert  \leq r,$$ then 
$$\vert \vert V_{p^{\mu_1}}-V_{p^{\mu_2}} \vert \vert_{\infty}\leq r \sup_{x \in [-L,L]}\E_x ( \tau_{-L} \wedge \sigma_L) , $$ 
with $V_{p^{\mu_1}},V_{p^{\mu_2}}$ the $(\epsilon,p^{\mu_1}) , \ (\epsilon, p^{\mu_2})$ value functions respectively.
\end{proposition}

\begin{proof}
Fix $x \in \mathbb{R}$. Observe $V_{p^{\mu_1}}(x)=V_{p^{\mu_2}}(x)$ if $x \notin (-L,L)$. Therefore we can assume $x \in (-L,L)$. Assume $V_{p^{\mu_1}} (x) \leq V_{p^{\mu_2}}(x)$. On one hand, using the fact for $i,j \in  \lbrace 1,2 \rbrace, \ \tau^{\ast}_{a_{ p^{\mu_i},\epsilon}} \wedge \sigma^{\ast}_{b_{p^{\mu_j},\epsilon}}$ is a stopping time smaller or equal than the first exit of the interval $[-L,L]$:

\begin{multline}
\vert M_{x,p^{\mu_1}}(\tau^{\ast}_{a_{p^{\mu_i},\epsilon}},\sigma^{\ast}_{b_{p^{\mu_j},\epsilon}})- M_{x,p^{\mu_2}}(\tau^{\ast}_{a_{p^{\mu_i},\epsilon}},\sigma^{\ast}_{b_{p^{\mu_j},\epsilon}}) \vert \\
= \left\vert  \E_{x} \left( \int_0^{\tau^{\ast}_{a_{p^{\mu_i},\epsilon}}\wedge \sigma^{\ast}_{b_{p^{\mu_j},\epsilon}}}( c'(X_{s},p^{\mu_1})-c'(X_s,p^{\mu_2}))e^{-\epsilon s}ds  \right) \right\vert \\ \leq r \sup_{x \in [-L,L]}\E_x ( \tau_{-L} \wedge \sigma_L).
\end{multline}

Therefore, 
\begin{align*}
V_{p^{\mu_2}}(x) & -  r \sup_{x \in [-L,L]}\E_x ( \tau_{-L} \wedge \sigma_L)
\\& = M_{x,p^{\mu_2}}(\tau^{\ast}_{a_{p^{\mu_2},\epsilon}},\sigma^{\ast}_{b_{p^{\mu_2},\epsilon}}) -  r \sup_{x \in [-L,L]}\E_x ( \tau_{-L} \wedge \sigma_L) 
\\& \leq M_{x,p^{\mu_2}}(\tau^{\ast}_{a_{p^{\mu_2},\epsilon}},\sigma^{\ast}_{b_{p^{\mu_1},\epsilon}}) -  r \sup_{x \in [-L,L]}\E_x ( \tau_{-L} \wedge \sigma_L)
\\ & \leq  M_{x,p^{\mu_1}}(\tau^{\ast}_{a_{p^{\mu_2},\epsilon}},\sigma^{\ast}_{b_{p^{\mu_1},\epsilon}})  \leq  M_{x,p^{\mu_1}}(\tau^{\ast}_{a_{p^{\mu_1},\epsilon}},\sigma^{\ast}_{b_{p^{\mu_1},\epsilon}})  =V_{p^{\mu_1}}(x).  
\end{align*}
\end{proof}

\begin{lemma}\label{L:convergencetotalvariation}
If $\lbrace (a_n,b_n) \rbrace_n$ is a sequence that converges to $(a,b)$, with $a\leq 0 \leq b$, then $ p^{a_n,b_n} \rightarrow p^{a,b}$ when $n \to \infty$.
\end{lemma}
\begin{proof}
First we assume $a \neq b$. We can assume that $a_n \neq b_n$ for every $n$ (it is true for $n$ big enough).
 Then observe, due to the continuity of $f$:
\begin{equation*}
    \lim_{n \to \infty} \E \left(f(X_{\infty}^{a_n,b_n})-f(X_{\infty}^{a,b})  \right)=  \lim_{n \to \infty}\lim_{T \to \infty}\frac{1}{T} \int_0^T \E \left(f(X_{s}^{a_n,b_n})-f(X_{s}^{a,b} ) \right) ds=0
\end{equation*}
 
For the case $a=0=b$, simply observe that the measures $X^{a_n,b_n}_{\infty}$ have support on the interval $[a_n,b_n] $ and $a_n \leq 0 \leq b_n$.
\end{proof}

We need a technical result that will allow us to control the increment of the function $V_p$
\begin{proposition}\label{Ch4P:MFGquaternion}
    Let us consider the quaternion $b_1<b_2<b_3<b_4$ then $$\inf_{x \in [b_2,b_3]}\E_x (\tau(b_1) \wedge \sigma(b_4) )> 0.$$
\end{proposition}
\begin{proof}
    Assume by contradiction that there is a sequence $\lbrace x_n \rbrace \subset [b_2,b_3]$ converging to $x$ such that 
    $$\lim_{n \to \infty} \E_{x_n} (\tau(b_1)\wedge \sigma(b_4)) =0 . $$
    We can assume (taking a subsequence if necessary) that the expected value decreases to zero. Therefore, we deduce $\tau(b_1-x_n) \wedge \sigma(b_4-x_n)  $ decreases to zero $\mathbf{P}$-almost surely. This implies $\tau(b_1-x) \wedge \sigma(b_4 -x) $ is equal to zero $\mathbf{P}$ almost surely which is a contradiction because it is the fist exit from an open interval containing zero.   
\end{proof}

We have enough properties to work with the map that gives the best response.
\begin{definition}
For a fixed $\epsilon>0$, define the function $F: [-L,0] \times [0,L] \rightarrow [-L,0] \times [0,L] $ as:
$ F(a,b)=(a^{\ast}_{p^{a,b},\epsilon},b^{\ast}_{p^{a,b},\epsilon}) $. Moreover we denote $F_1(a,b)$ and $F_2(a,b)$ as the projections of $F$ in the first and second coordinate respectively.

\end{definition}

\begin{lemma}\label{CH4L:ContinuityDynkinsolution}
The function $F$ is continuous. 
\end{lemma}
\begin{proof}
We will only prove $F_1$ is continuous (the proof of the continuity of $F_2$ is analogue).
For $(a,b) \in [-L,0] \times [0,L] $, take $x \in (a^{\ast}_{p^{a,b},\epsilon},b^{\ast}_{p^{a,b},\epsilon})$. Observe $ V_{p^{a,b}}(x)< q_d$. Take a sequence $\lbrace (a_n , b_n) \rbrace_{n \geq 0} \subset [-L,0] \times [0,L] $ converging to $(a,b)$. Observe, due to Lemma \ref{L:convergencetotalvariation} and Proposition \ref{P:contintuityinP} we have
$$\limsup_{n  \to \infty} V_{p^{a_n,b_n}}(x)< q_d. $$
Therefore $\limsup_{n \to \infty} F_1(a_n,b_n) \geq \liminf_{n \to \infty} F_1(a_n,b_n) \geq  b^{\ast}_{p^{a,b},\epsilon}$. Let us assume by contradiction that  $\limsup_{n  \to \infty} F_1(a_n,b_n) =\hat{b} >  b^{\ast}_{p^{a,b},\epsilon}$.  Set $b_1<b_2<b_3 $ such that $(b_1,b_3) \subset (b^*_{p_{a,b},\epsilon},\hat{b})$. Observe, using Lemma \ref{L:MEANFIELDDynkinLipschitz} that for $N$ big enough (using the fact $c_{xx}(x,y)>0$ whenever $x \neq 0$):
\begin{multline*}
\liminf_{h \to 0^+}{V_{p^{a_n,b_n}}(x+h)-V_{p^{a_n,b_n}}(x)\over h} \\
\geq  \E \left( \int_0^{\tau(a^*_{p_{a^n,b^n},\epsilon})_{x}\wedge \sigma(b^*_{p_{a^n,b^n},\epsilon})_{x+h}}\left( c_{xx}(X_s,p_{a^n,b^n})  \right)e^{-\epsilon s}\,ds\right) \\
 \geq  \liminf_{h \to 0^+} \E\int_0^{\tau(b^*_{p_{a,b},\epsilon})_{x}\wedge \sigma(b_3)_{x}}\left(\inf_{(h,y)\in   (-b^*_{p_{a,b},\epsilon},b^*_{p_{a,b},\epsilon})^c \times \mathbb{R}} c_{xx}(h,y) \right)e^{-\epsilon s}\,ds.
\end{multline*}
for every $x \in (b_1.b_2)$. Due to the fact 
$$ \inf_{x \in (b_1,b_2)}\E(\tau(b^*_{p_{a,b},\epsilon})_{x}\wedge \sigma(b_3)_{x})>0, $$ we deduce that there is a $\delta>0$ such that for every $x \in (b_1,b_2)$ and $n$ big enough

\begin{multline}\label{eq:boundforderivatecontinuity}
\liminf_{h \to 0^+}{V_{p^{a_n,b_n}}(x+h)-V_{p^{a_n,b_n}}(x)\over h} 
\\ \geq  \E \left( \int_0^{\tau(b^*_{p_{a,b},\epsilon})_{x}\wedge \sigma(b_3)_{x}}\left(\inf_{(h,y)\in  (-b^*_{p_{a,b},\epsilon},b^*_{p_{a,b},\epsilon})^c \times \mathbb{R} } c_{xx}(h,y)  \right)e^{-\epsilon s}\,ds\right)=: \delta .
\end{multline}
Observe, due to Proposition \ref{Ch4P:MFGquaternion} the constant $\delta$ is greater than zero.
Thus, using Lemma \ref{L:convergencetotalvariation} and Proposition \ref{P:contintuityinP}, we get:
\begin{equation*}
0=V_{p^{a,b}}(b_2)-V_{p^{a,b}}(b_1)= \lim_{n \to \infty}V_{p^{a_n,b_n}}(b_2)-V_{p^{a_n,b_n}}(b_1) \geq \delta(b_2-b_1), 
\end{equation*}
arriving to a contradiction. Therefore $F_1$ is continuous. With an analogue argument we can prove $F_2$ is continuous, concluding the proof of the lemma.
\end{proof}

\begin{proof}[Proof of Theorem \ref{T:MEANFIELDDISCOUNTEDPROBLEMSOLUTION}]
Due to Proposition \ref{P:BOUNDSABWITHMEANFIELD} and Lemma \ref{CH4L:ContinuityDynkinsolution}, we are in the hypothesis of Brouwer fixed-point Theorem. Thus, using the notation of Proposition \ref{P:BOUNDSABWITHMEANFIELD} we have a fixed point $(a^{\ast},b^{\ast})$ of the function $F$. Due to Theorem \cite[Theorem 1]{MO} we deduce that $(a^{\ast},b^{\ast})$ is an $\epsilon$-discounted equilibrium.
\end{proof} 

\section{Mean field equilibrium for the ergodic problem}\label{C43:MFGequilibriumforergodic}
To prove Theorem \ref{T:ErgodicProblemsolutionMFG}, to show that the optimal strategies converge to an optimal strategy, first we need to show that 
$$\lim_{\epsilon \to 0} \epsilon J_{\epsilon} (x,U^{a,b},D^{a,b},p) \to J(x,U^{a,b},D^{a,b},p). $$
For that endeavor we need to use theory of regenerative processes. For the reader's convenience, we write
some known results of \emph{renewal theory}.

\begin{definition}\label{ChTools:renewalprocess}\cite[Chapter V, Section 1] {ASMUSSENQUEUES}, 
Let $0 \leq S_0 \leq S_1< S_2 < \dots $ be \emph{the times of occurrence of some phenomenon} (in this article a stopping time) and $Y_n=S_n-S_{n-1}$, $Y_0=S_0$. Then $\lbrace S_n\rbrace_{n \in \mathbb{N}}$ (with the zero included) is called renewal process if $Y_0,Y_1, \dots $ are independent and $Y_1, Y_2,\dots $ (but not necessarily $Y_0$) have the same distribution. 
\end{definition}

\begin{definition}\label{ChTools:accumulativeprocess}\cite[Chapter VI, Section 3]{ASMUSSENQUEUES}
  
A real–valued process $\lbrace Z_t \rbrace$ is called cumulative if $Z_0=0$ and there exists
a renewal process $\lbrace S_n \rbrace $ such that for any $n$, $ \lbrace Z_{S_n +t }-Z_{S_n} \rbrace$ is independent of $S_0, \ S_1, \dots ,S_n $ and $\lbrace Z_t \rbrace_{t < S_n},$ and for every $t\geq 0, \ n,m \in \mathbb{N}$:   
$Z_{S_n +t }-Z_{S_n}=Z_{S_m +t }-Z_{S_m}  \text{ in law.} $ 
\end{definition}

\begin{theorem}\label{ChToolsT:asmussen}\cite[Chapter VI, Theorem 3.1]{ASMUSSENQUEUES}
Suppose $\lbrace S_n \rbrace$ is a renewal process and $Z_t$ is an accumulative process. Moreover assume $S_0=0, \ \E (S_2-S_1)< \infty, \ \E \vert Z_{S_1} \vert < \infty$. Then 
$$\lim_{t \to \infty} \frac{Z_t}{t} =  \frac{\E(Z_{S_1})}{\E(S_2-S_1)} \ \text{a.s} \quad \text{if and only if } \E \left(\max_{0 \leq t \leq S_1}\vert Z_t \vert \right) < \infty. $$
\end{theorem}

\subsection{Ergodic results}
The main tool that helps us in this section is Theorem \ref{ChToolsT:asmussen} and to use that theorem, we need to define an adequate renewal process and an adequate accumulative process.
 We define $\lbrace \tau_n \rbrace_n $ as:
\begin{align*}
&\tau_0 = \inf \lbrace t \geq 0, X^{0,b}_t=0, \sup_{0 \leq s \leq t} X^{0,b}_s=b \rbrace , \\
&\tau_{n+1}  = \inf \lbrace t \geq \tau_n,\ X^{0,b}_t=0, \sup_{\tau_n  \leq s \leq t} X^{0,b}_s=b \rbrace  .
\end{align*}
Notice that $\lbrace \tau_n \rbrace_n$ is a renewal process and every $\tau_n$ is a stopping time. Moreover: 
\begin{equation}\label{eq:renewalbounded}
 \E_x (\tau_0) \leq \E_x(\tau_{n+1}-\tau_n)< \infty  \text{ for all } n \geq 1, \ x \in \mathbb{R}.
 \end{equation}

We proceed to study the abelian limit of the costs.

\begin{proposition}\label{P:convergenceepsilon}
If $X$ has bounded (unbounded) variation and $a \leq b$ ($a <b$), $x \in \mathbb{R}$, then:

$$\lim_{\epsilon \to 0}\epsilon J_{\epsilon}(x,U^{a,b},D^{a,b},X_{\infty}^{a,b})=J(x,U^{a,b},D^{a,b},X^{a,b}_{\infty}) $$
\end{proposition}
\begin{proof}
First of all, notice that the next arguments hold if $a=b$ (when the process has bounded variation). Moreover, it is enough to prove (even if $x\notin [a,b]$):
\begin{align}\label{Eq:EPSILON1BIGCONVERGENCEwithc}
&\lim_{\epsilon \to 0}\epsilon \E_x \left( \int_0^{1/{\epsilon}}c(X^{a,b}_s,p^{a,b}) ds  - \int_0^{\infty}\left(c(X^{a,b}_s,p^{a,b}) \right)e^{-\epsilon s}ds\right)=0,
\\
 \label{Eq:2EPSILONBIGCONVERGENCEwithLinear}
&\lim_{\epsilon \to 0}\epsilon \E_x \left(\int_{0}^{1/\epsilon}q_u  d(D^{a,b}_s)  -\int_{0}^{\infty}q_ue^{-\epsilon s} d(D^{a,b}_s) \right)=0, 
\end{align}
Observe that we can assume $x=a$  (which implies $\tau_0=0$ and there is no first jump) due to the strong markov property. By noticing that if the process starts at $a$, $X^{a,b}_t$ and $D^{a,b}_t$ are equal in law to the processes starting at zero $X^{0,b-a}_t+a$ and $D^{0,b-a}_t$ for every $t$ respectively, we can also assume $a=0$.
To prove \eqref{Eq:EPSILON1BIGCONVERGENCEwithc} and \eqref{Eq:2EPSILONBIGCONVERGENCEwithLinear}, denote $S_n:= \tau_n,$  for all $n \in \mathbb{N}$.  Furthermore notice that the processes
\begin{equation*}
(Z_1)_t := \int_0^{t}\left(c(X^{0,b}_s+a,p^{0,b}) \right)ds , \quad (Z_2)_t : =\int_{0}^{t}q_u d(D^{0,b}_s)  
\end{equation*}
are cumulative. We proceed to prove that both processes are in the hypothesis of Theorem \ref{ChToolsT:asmussen}. Firstly, using the continuity of $c$ and \eqref{eq:renewalbounded}, we deduce $Z_1$ is in the hypothesis of the mentioned Theorem.\\
Secondly notice;
\begin{equation*}
\E \left(\max_{0 \leq t \leq \tau_1} \vert (Z_2)_t \vert \right)\leq q_u\E D^{0,b}_{\tau_1} .
\end{equation*}
Therefore, to prove that $Z_2$ is in the hypothesis of the Theorem, it is enough to prove 
\begin{equation*}
\E D^{0,b}_{\tau_1}< \infty.
\end{equation*}
Observe, using the inequality obtained from \cite[Theorem 6.3, equation (70)]{AAGP}:
\begin{multline*}
    b D^{0,b}_t  \leq b^2+x^2  +2\int_{0^+}^t X^{0,b}_{s^-}dX_s+\frac{\sigma^2}{2}t \\ + \sum_{s \leq t} \left(\mathbf{1}_{\Delta X_s \geq b} (2 \Delta X_s b+b) 
    + \mathbf{1}_{\Delta X_s \leq -b}(b^2 -2b \Delta X_s)+ \mathbf{1}_{\vert \Delta X_s \vert <b } \Delta X_s^2 \right),
\end{multline*}
the fact $X$ has finite mean and  $\E (\tau_1)<\infty$ we deduce \eqref{eq:FINITUDEtau} holds and $Z_2$ is in the hypothesis of the Theorem. \\
 To finish the proof of the Proposition, we study the second integral of each equation \eqref{Eq:EPSILON1BIGCONVERGENCEwithc}  and \eqref{Eq:2EPSILONBIGCONVERGENCEwithLinear}. More precisely, observe:
 
\begin{align*}
&\lim_{\epsilon \to 0}\epsilon \E_a \left( \int_0^{\infty} c(X_s^{a,b},p^{a,b})e^{-\epsilon s}ds \right)
\\ &=\lim_{\epsilon \to 0} \epsilon \sum_{n=0}^{\infty}  \left(\E_a e^{-\epsilon \tau_1} \right)^n   \E_a \left( \int_0^{\tau_1}c(X_s^{a,b},p^{a,b}) e^{-\epsilon s}ds\right) \\ &=\lim_{\epsilon \to 0} \epsilon \sum_{n=0}^{\infty}  \left(\E_a e^{-\epsilon \tau_1} \right)^n   \E_a \left( \int_0^{\tau_1}c(X_s^{a,b},p^{a,b}) ds\right)
= \lim_{\epsilon \to 0} \epsilon \frac{\displaystyle \E_a \left( (Z_1)_{\tau_1} \right)}{\displaystyle 1-\E_a(e^{-\epsilon \tau_1})}\\ &= \lim_{\epsilon \to 0} \frac{\displaystyle \E_a \left( (Z_1)_{\tau_1} \right)}{\displaystyle \E_a \left( \left(1-e^{-\epsilon \tau_1}\right) \frac{\tau_1}{\epsilon \tau_1} \right)}= \frac{\displaystyle  \E_a \left( (Z_{1})_{\tau_1} \right) }{\E_a \tau_1},
\end{align*}
where in the last equality the dominated convergence Theorem has been used. Therefore, using Theorem \ref{ChToolsT:asmussen}, we deduce the equation \eqref{Eq:EPSILON1BIGCONVERGENCEwithc} holds. A similar reasoning can be used to deduce that \eqref{Eq:2EPSILONBIGCONVERGENCEwithLinear}  holds, concluding the proof of the proposition.\\
\end{proof} 

\begin{lemma}\label{L:continuitypowerful}
If $X$ has bounded (unbounded) variation and $a \leq b$ ($a <b$), $\lbrace (a_n, b_n,\epsilon_n) \rbrace_n$ is a sequence that converges to $(a,b,0)$ when $n \to \infty$ and $a_n \leq b_n, \ (a_n<b_n)$ for all $n$,  then:

$$\lim_{n \to \infty}\epsilon_n J_{\epsilon_n}(x,U^{a_n,b_n},D^{a_n,b_n},X^{a_n,b_n}_{\infty})=J(x,U^{a,b},D^{a,b},X^{a,b}_{\infty}) $$
for all $x \notin \lbrace a,b \rbrace $.
\end{lemma}
\begin{proof}
First, we assume $x \in ( a,b) $. Due to Proposition \ref{P:convergenceepsilon}, it is enough to prove:
\begin{multline}\label{Eq:1BIGCONVERGENCEwithc}
\lim_{n \to \infty}\epsilon_n \E_x \left( \int_0^{\infty}\left(c(X^{a,b}_s,p^{a,b})e^{-\epsilon_n s} \right)ds \right. \\ \left. -  \int_0^{\infty}\left(c(X^{a_n,b_n}_s,p^{a_n,b_n}) \right)e^{-\epsilon_n s}ds\right)=0 
\end{multline}
 and \begin{equation}\label{Eq:2BIGCONVERGENCEwithLinear}
\lim_{n \to \infty}\epsilon_n \E_x \left(\int_{0}^{\infty}q_u e^{-\epsilon_n s} d(U^{a,b}_s)  -\int_{0}^{\infty}q_u e^{-\epsilon_n s} d(U^{a_n,b_n}_s) \right)=0. 
\end{equation}
Equality \eqref{Eq:1BIGCONVERGENCEwithc} holds due to the fact $c$ is continuous and Lemma \ref{L:convergencetotalvariation}. \\
Equation \eqref{Eq:2BIGCONVERGENCEwithLinear} is deduced from \cite[Proposition 10]{MO} and integration by parts. 
For the case $x<a$, observe $x< a_n$ for $n$ big enough. Moreover the initial jump can be ommited in the limit because $\epsilon_n$ goes to zero. Finally, by translating the process we observe that it is enough to prove

\begin{multline*}
\lim_{n \to \infty}\epsilon_n \E \left( \int_0^{\infty}c(a+X^{0,b-a}_s,p^{a,b})e^{-\epsilon_n s} ds\right. \\ \left. - \int_0^{\infty}c(a_n+X^{0,b_n-a_n}_s,p^{a_n,b_n}) e^{-\epsilon_n s}ds\right)=0
\end{multline*}
and
 \begin{equation*}
\lim_{n \to \infty}\epsilon_n \E \left(\int_{0}^{\infty} q_u e^{-\epsilon_n s} d(U^{0,b-a}_s)  -\int_{0}^{\infty}q_u e^{-\epsilon_n s} d(U^{0,b_n-a_n}_s) \right)=0, 
\end{equation*}

Following the same line of reasoning as the case $x\in (a,b)$, it can be proven that the three limits hold. Finally the case $x>b$ is obviously analogue to the case $x<a$, thus the proof of the lemma is concluded.

\end{proof}
\subsection{Proof of Theorem \ref{T:ErgodicProblemsolutionMFG}}
We have enough results to prove Theorem \ref{T:ErgodicProblemsolutionMFG}.
\begin{proof}[Proof of Theorem \ref{T:ErgodicProblemsolutionMFG}] Using the notation of Proposition \ref{P:BOUNDSABWITHMEANFIELD},
take a sequence \newline $\lbrace (a^{\ast}_{ \epsilon_n},b^{\ast}_{\epsilon_n})\rbrace_{n \in \mathbb{N}} \subset [-L,L]^2 $ such that  $(a^{\ast}_{\epsilon_n},b^{\ast}_{\epsilon_n})$ is an $\epsilon_n$ discounted equilibrium for every $n$ and a couple $(a^{\ast},b^{\ast}) \in [-L,L]^2$ satisfying 
$$ (\epsilon_n, (a^{\ast}_{ \epsilon_n},b^{\ast}_{\epsilon_n})) \to (0,(a^{\ast},b^{\ast})), \ \text{when} \ n \to \infty.$$ This sequence exists due to Proposition \ref{P:BOUNDSABWITHMEANFIELD}. Now we assume $x \notin \lbrace a^{\ast},b^{\ast}\rbrace$. To prove $\rm{(i)}$ observe it is enough to prove
\begin{align}
\label{eq:ergodic1}&G(x,X_{\infty}^{a^{\ast},b^{\ast}})-\lim_{n \to \infty}  \epsilon_n G_{\epsilon_n}(x,X_{\infty}^{a^{\ast},b^{\ast}})=0, \\
\label{eq:ergodic2}&\lim_{n \to \infty}\left( \epsilon_n G_{\epsilon_n}(x,X_{\infty}^{a^{\ast},b^{\ast}}) - \epsilon_n G_{\epsilon_n}(x,X_{\infty}^{a^{\ast}_{\epsilon_n},b^{\ast}_{\epsilon_n}}) \right) =0,\\
\label{eq:ergodic3}&\lim_{n \to \infty} \epsilon_n G_{\epsilon_n}(x,X_{\infty}^{a^{\ast}_{\epsilon_n},b^{\ast}_{\epsilon_n}})- J(x,X_{\infty}^{a^{\ast},b^{\ast}},U^{a^{\ast},b^{\ast}},D^{a^{\ast},b^{\ast}})=0.
\end{align}
The limit \eqref{eq:ergodic1} is deduced from \cite[Theorem 2]{MO}.  To prove that the second limit \eqref{eq:ergodic2} holds, observe that for every $ X_{\infty}^{\mu_1} , X_{\infty}^{\mu_2} \in \mathcal{P}^{\infty},A\leq B  , \ A,B \in [-L,L], \ \epsilon>0 $ (the inequality strict if the process has unbounded variation) :
\begin{multline*}
\epsilon J_{\epsilon}(x,U^{A,B},D^{A.B},X_{\infty}^{\mu_1} )-\epsilon J_{\epsilon}(x,U^{A,B},D^{A.B},X_{\infty}^{\mu_2} ) \\ =\epsilon \E_x \left(\int_0^{\infty} e^{-\epsilon s}\left( c(X_s^{A,B},X_{\infty}^{\mu_1} )-c(X_s^{A,B},X_{\infty}^{\mu_2} ) \right)ds  \right) .
\end{multline*} 
Therefore:
\begin{equation}\label{eq:ergodic4}
\epsilon_n G_{\epsilon_n}(x,p^{a^{\ast},b^{\ast}}) - \epsilon_n G_{\epsilon_n}(x,p^{a^{\ast}_{\epsilon_n},b^{\ast}_{\epsilon_n}}) \leq 2 \sup_{y \in [-L,L]} \vert c(y,p^{a^{\ast},b^{\ast}})-c(y,p^{a^{\ast}_{\epsilon_n},b^{\ast}_{\epsilon_n}}) \vert . 
\end{equation}
Thus from the continuity of $c$ and Lemma \ref{L:convergencetotalvariation}, we conclude that the limit \eqref{eq:ergodic2} holds. Finally, the limit \eqref{eq:ergodic3} is deduced from Theorem \cite[Theorem 1]{MO}  and Lemma \ref{L:continuitypowerful}. \\ 
To prove $\rm{ii)}$, due to $\rm{i)}$, we use the limits in \eqref{eq:ergodic1} and \eqref{eq:ergodic2}. \\
For the case $x \in \lbrace a^\ast,b^\ast \rbrace$ simply observe $$ \left| G_{\epsilon_n}(y,X^{a^\ast,b^\ast}_{\infty})- G_{\epsilon_n}(z,X^{a^\ast,b^\ast}_{\infty}) \right| \leq  (q_d+q_u) \left| z-y \right| , \ \text{ for every } y,z \in \mathbb{R}$$ and the fact that the function $G(\cdot , X_{\infty}^{a^{\ast},b^{\ast}})$ is constant.
\end{proof}

\section{Examples}\label{C43:Examples}
We provide two examples, one for the discounted problem and one for the ergodic problem.
\subsection{Discounted MFG problem for a Compound Poisson process with two-sided exponential jumps}\label{ExamplePoissonMFG}
In this example $q:=q_d=q_u$ and $c(x,y)=x^2 h(y)$ with $h$ a convex and strictly positive function with minimum at zero.

We consider a compound Poisson process $X=\{X_t\}_{t\geq 0}$ with double-sided exponential jumps, given by
\begin{equation*}
X_t=x-\sum_{i=1}^{N^{(1)}_t}Y^{(1)}_i+\sum_{i=1}^{N^{(2)}_t}Y^{(2)}_i,
\end{equation*}
where $\{N^{(1)}_t\}_{t\geq 0}$ and 
$\{N^{(2)}_t\}_{t\geq 0}$ 
are two Poisson processes with respective positive intensities $\lambda_1, \ \lambda_2$;
$\{Y^{(1)}_i\}_{i\geq 1}$ and $\{Y^{(2)}_i\}_{i\geq 1}$ 
are two sequences of independent exponentially distributed random variables with respective positive parameters 
$\alpha_1,\alpha_2$. Furthermore, we assume $\E X_1=-\lambda_1/\alpha_1 +\lambda_2/\alpha_2=0$. For the best response function, for every fixed $y$ we are in the search of equilibrium points of
\begin{multline}\label{eq:CompoundPoissonDynkin1}
\inf_{\sigma(b)}\sup_{\tau(a)}\E   \left(\int_0^{\tau(a) \wedge \sigma(b)}  2X_s h(y) e^{-\epsilon s}ds \right. \\ \left.  +q e^{-\epsilon \tau(a)} \mathbf{1}_{\tau(a) \leq \sigma(b) }  
     -q e^{-\epsilon \sigma(b)} \mathbf{1}_{\sigma(b) <\tau(a)}   \right). 
\end{multline}
Using integration by parts, due to $X$ being a martingale we have:
$$h(y)\E e^{-\epsilon (\tau(a) \wedge \sigma(b))}X_{\tau(a)\wedge \sigma(b)}=-\epsilon h(y) \E \left(\int_0^{\tau(a)\wedge \sigma(b)}X_s e^{-\epsilon s}ds \right). $$
Therefore, we can rewrite \eqref{eq:CompoundPoissonDynkin1} as 
\begin{multline}\label{eq:CompoundPoissonDynkin2}
\frac{2h(y)}{\epsilon}\inf_{\sigma(b)}\sup_{\tau(a)}\E   \left(   (-X_{\tau(a)\wedge \sigma(b)})  e^{-\epsilon (\tau(a)\wedge \sigma(b)} \right. \\ \left. +\frac{\epsilon}{2h(y)} q  e^{-\epsilon \tau(a)} \mathbf{1}_{\tau(a) \leq \sigma(b) }  
     -\frac{\epsilon}{2h(y)} q e^{-\epsilon \sigma(b)} \mathbf{1}_{\sigma(b) <\tau(a)}   \right). 
\end{multline}

Let $\delta^y:=\frac{\epsilon}{2h(y)} q$. Using the fact that $-X$ is also a compound poisson process with exponential jumps, we can use the results of \cite{ASPISOSADECKI} (notice the indexes are inverted as we work with $-X$), that is the best response points $(a^y, b^y )$ are the solutions of the equations:

\begin{align}\label{eq:Bestresponse}
    &a^y=  -\delta^y -E_I+ F_I \frac{\displaystyle E_S e^{r_I (b^y-a^y)} -E_I G_S }{\displaystyle e^{(r_I +r_S)(b^y-a^y)} -G_I G_S },     \nonumber    \\
    &b^y=  \delta^y +E_S+ F_S \frac{-E_I e^{r_S(b^y-a^y)}+E_SG_I}{e^{(r_I+r_S)(b^y-a^y)}-G_I G_S},
\end{align}
with
\begin{align*}
    & E_I = \frac{1-\pi_I}{r_I}, & E_S = \frac{1-\pi_S}{r_S},  \\
    &F_I=\frac{r_1+r_2}{ r_I+\pi_I r_S}, &  F_S= \frac{r_I +r_S}{\pi_Sr_I+r_S}, \\
    &G_I=\frac{(1-\pi_I)r_S}{ r_I+\pi_I r_S}, &  G_S= \frac{(1-\pi_S)r_I}{\pi_S r_I +r_S}, \\
    & 
    \pi_I = r_I/\alpha_2,  & \pi_S=r_S/\alpha_1.
\end{align*}
and $-r_I<0<r_S$ the roots of the equation  $\phi_{-X}(z)=\epsilon $ (with $\phi_{-X} $ the characteristic exponent of $-X$), that is the roots of:
$$(\epsilon+\lambda_2+\lambda_2)z^2+(\alpha_2 (\lambda_1+\epsilon)-\alpha_1 (\lambda_2+\epsilon))z -\epsilon \alpha_2 \alpha_1=0   . $$
  Thus, using \eqref{eq:CompoundPoissonDynkin2} and $\eqref{eq:Bestresponse}$, we conclude the MFG equilibrium points $(a^{\ast},b^{\ast})$ are the roots of the equation:
\begin{align*}
    &a^{\ast}=  -\frac{\displaystyle \epsilon q}{2 h\left( \E f(X^{a^\ast,b^\ast}_{\infty}) \right) } -E_I+ F_I \frac{\displaystyle E_S e^{r_I (b^\ast-a^\ast)} -E_I G_S }{\displaystyle e^{(r_I +r_S)(b^\ast-a^\ast)} -G_I G_S },     \nonumber    \\
    &b^{\ast}=  \frac{\displaystyle \epsilon q}{2 h \left( \E f(X^{a^\ast,b^\ast}_{\infty})\right) } +E_S+ F_S \frac{-E_I e^{r_S(b^\ast-a^\ast)}+E_SG_I}{e^{(r_I+r_S)(b^\ast-a^\ast)}-G_I G_S},
\end{align*}
If, for example, $f(x)=\vert x \vert, \ c(x,y)=x^2 (\vert y \vert +1)$. What is left to    we need to compute the stationary distribution, see \cite[Chapter V]{AAGP}, $\pi^{a,b}$ of $X^{a,b}$, where:
\begin{equation}\label{E:stationary}
\pi^{a,b}[x,b]=\mathbf{P}(X_{\eta_{[x-a-b,x-a)^c}}\geq x-a ), \ x\in[a,b] ,
\end{equation}
where $\eta_{[x-a-b,x-a)^c}$ denotes the first entry of $X$ to the set $ [x-a-b,x-a)^c$
For this process, as $X$ is a martingale and the loss of memory property of the jumps, we have:
\begin{align*}
    &1= \P (X_{\eta_{[x-a-b,x-a)^c}}\geq x-a  )  + \P (X_{\eta_{[x-a-b,x-a)^c}}<x-a), \\
    & 0=\P (X_{\eta_{[x-a-b,x-a)^c}}\geq x-a  )\left(\int_0^{\infty}\alpha_2 e^{- \alpha_2 u}(u+x-a) du \right) \\   & +\P (X_{\eta_{[x-a-b,x-a)^c}}<x-a) \left(\int_{-\infty}^{0}\alpha_1e^{\alpha_1 u}(u +x-a-b)du \right) , 
    \end{align*}
we obtain that the stationary measure in \eqref{E:stationary} is of the form:
$$ \pi^{a,b}(dx)=\delta_a(dx) \frac{\displaystyle 1/\alpha_2}{b+1/\alpha_2+1/\alpha_1}+\frac{\displaystyle dx}{b+1/\alpha_2+\alpha_1} + \delta_b (x) \frac{\displaystyle a+1/\alpha_2}{b+1/\alpha_2+1/\alpha_1},$$
thus

$$\E X_{\infty}^{a,b} = \frac{\displaystyle a/\alpha_2}{b+1/\alpha_2+1/\alpha_1}+\frac{\displaystyle b^2/2 -a^2/2}{b+1/\alpha_2+\alpha_1} +  \frac{\displaystyle b(a+1/\alpha_2)}{b+1/\alpha_2+1/\alpha_1} $$

\begin{remark}\label{Remark:Nonuniqueness}
    On behalf of the uniqueness, due to only requirement of $f$ is to be continuous, on can construct examples with many solutions. For example, under the hypotheses of this subsection:
    \begin{itemize}
        \item[\rm{i)}] Consider a pair of Dynkin games defined as before except for the fact that instead of $\delta^y$ we take a couple of different positive constants $\delta_1,\delta_2$  (without MFG component) whose solutions are given in \eqref{eq:Bestresponse}.
        \item[\rm{ii)}] For $\delta_1$ and $\delta_2$, let us name the solutions $(a_1,b_1)$ and $(a_2,b_2)$ respectively.
        \item[\rm{iii)}]
        Now, consider $h$ and $f$ a pair of continuous positive functions such that 
        $$h \left( \E f(X^{a_1,b_1}_{\infty})   \right) =\delta_1, \qquad h \left( \E f(X^{a_2,b_2}_{\infty})   \right) =\delta_2 . $$
        \item[\rm{iv)}] For this case, both $(a_1,b_1)$ and $(a_2,b_2)$ are MFG equilibrium points for the problem defined in this subsection.
        \item[\rm{v)}] For example, if $f(x)=x,  \ h(y)=0.01 + e^y \vert \cos(y) \vert$ and $\alpha_1=1, \ \alpha_2=2, \lambda_2=3, \ \epsilon=0.1, q=0.5$ (with $\E X_1=0 $), we have at least two MFG equlibrium points whose value is approximately :
        $$( -5.846 , 6.038), \ (-0.581 , 0.810) .$$
    \end{itemize}
\end{remark}

\subsection{Ergodic MFG for Strictly stable process}
In this case $c(x,y)=x^2 (1+\vert y \vert ), \ q_d=q_u=q$ and $f(y)=y^2$.
This example adds a MFG component to the example provided in \cite[5.4]{MO}.
Here the L\'evy process $X$ is strictly $\alpha$- stable  with parameter $\alpha \in (1,2), 0<c^+<c^-$.
In other words $X$ is a pure jump process with finite mean and triplet $(0,0,\Pi)$, with jump measure
\[ 
\Pi(dx)= 
\begin{cases}
c^{+} x^{-\alpha -1}dx, & x>0, \\
c^{-} \vert x \vert^{-\alpha -1} dx ,& x<0.
\end{cases}
\]
The characteristic exponent is
\[
\phi(i \theta) \colon =\vert \theta \vert^{\alpha}(c^+ + c^-) \Big(1- i \text{sgn}(\theta)\tan \big(\frac{\pi \alpha}{2} \big) \frac{c^+ -c^-}{c^+ + c^-} \Big) .
\]
(see \cite[page 10]{KRI}). 
The stationary measure has Beta density $\pi^{0,d}(x)$ on $[0,d]$ with parameters
$(\alpha \rho, \alpha (1-\rho))$, i.e.
\begin{equation*}
\pi^{0,d}(x)=\frac1{d \beta(\alpha \rho, \alpha (1-\rho))} \left(1-\frac{x}{d} \right)^{\alpha \rho -1} 
\left(\frac{x}{d} \right)^{\alpha (1-\rho)-1}, 
\end{equation*}
where 
$$
\beta(u,v)=\int_0^1t^{u-1}(1-t)^{v-1}\,dt,
$$ 
is the Beta function and 
\[
\rho= \frac{1}{2}+ (\pi \alpha)^{-1} \arctan \bigg( \Big(\frac{c^+ - c^-}{c^+ + c^-} \Big)  \tan(\alpha \pi/2) \bigg). 
\]
From this
\begin{equation}\label{Ch3:Stablemoments}
    \int_0^dx\pi^{0,d}(x)\,dx=d\rho,\quad
\int_0^d(x-d\rho)^2\pi^{0,d}(x)\,dx=d^2{\rho(1-\rho)\over\alpha+1}.
\end{equation}

On the other hand, observe $\E_{\pi}(D^{0,d}_1)$  can be expressed explicitly
(see \cite[page 114]{AAGP}): 
\begin{multline}\label{E:stableProcesslossrate}
\E_{\pi}(D_1^{0,d})   = \frac{c^- \beta(2-\alpha \rho, \alpha \rho)+c^+ \beta(2-\alpha(1- \rho),\alpha (1-\rho))}{\beta(\alpha \rho, \alpha (1-\rho))\alpha (\alpha -1)(2-\alpha)}\frac{1}{d^{\alpha-1}}
    \\=\E_{\pi}(D_1^{0,1})\frac{1}{d^{\alpha-1}}.
\end{multline}
  Now, from the fact $\E X_1=0$, we have $\E_{\pi }U_1^{0,d}=\E_{\pi} D^{0,d}_1$ so the best response function $J$ is of the form (recall the starting variable $x$ can be ommited):

  \begin{equation*}
J((U^{a,b},D^{a,b}),X_{\infty}^{\eta})  =(1+p^{\eta} ) \E (X_{\infty}^{0,b-a}+a)^2+\frac{1}{d^{\alpha-1}}2q\E_{\pi}(D^{0,1}_1).   
\end{equation*}
Observe that the degenerate case $a=0=b$ can be discarded as the previous expression goes to $-\infty$ in that case.
Thus, by making the change of variable $d= b-a>0$ and differentiating respect to $a$ and \eqref{Ch3:Stablemoments} we deduce that for the best response  we have
\begin{equation*}
a=- \int_0^d x \pi^{0,d}(x)dx =-d\rho,
\end{equation*}
and the map to minimize with respect to $d$ is
\begin{equation*}
(1+p^\eta ) d^2{\rho(1-\rho)\over\alpha+1}+\frac{1}{d^{\alpha-1}}2q\E(D^{0,1}_1).    
\end{equation*}
By differentiation, the best response $(a^\ast,b^\ast)$ satisfies:
$$
b^\ast-a^\ast=d^*=\left(\frac{(\alpha^2-1)q \E(D^{0,1}_1)}{(1+p^\eta )\rho(1-\rho)}\right)^{1/(\alpha+1)}.
$$
Now, to find the equilibrium point, using \eqref{Ch3:Stablemoments}, we deduce
$$p=1+\E (X^{0,d}_{\infty}-d \rho)^2  = 1+ d^2 \frac{\rho (1-\rho)}{\alpha+1}. $$
Therefore we conclude that the points $(a^{\ast},b^{\ast})$ that define an ergodic MFG equlibrium are unique and are characterized by the equations:
\begin{align*}
   & a^{\ast}=-b^{\ast}\frac{\rho}{1-\rho} , \\
   & (b^*/(1-\rho))^{\alpha +1} \left(\frac{\alpha +1 +(b^*)^2 \rho/(1-\rho)}{\alpha +1}  \right)=\frac{(\alpha^2-1)q \E(D^{0,1}_1)}{\rho(1-\rho)}.
\end{align*}
With the change of variable $u:= (b^*)^{\alpha+1}$ we can solve the equation and obtain:
\begin{align*}
   & a^{\ast}=-b^{\ast}\frac{\rho}{1-\rho} , \\
   & b^*= \left(\frac{\displaystyle -\frac{1}{(1-\alpha)^{\alpha +1}} + \sqrt{\frac{1}{(\alpha -1)^{2\alpha +2} } +4 \frac{\rho}{(1-\rho)^{\alpha +2}} \frac{(\alpha^2-1)q \E(D^{0,1}_1)}{\rho(1-\rho)}    }}{\displaystyle 2 \frac{\rho}{(1-\rho)^{\alpha +2}} } \right)^{1/(\alpha+1)}.
\end{align*}

For example when we take the values $q=1/2, \ c^-=2, \ c^+=1, \ \alpha=3/2$ the MFG ergodic equilibrium $(a^{\ast},b^{\ast})\sim (-0.520,0.395)$.

\section{Relation to symmetric finite-player games }\label{C4:Nplayer}

In this section, 
we show that the assumption of stationary in the second variable implies that the equilibrium points of the MFGs are approximate  Nash equilibrium points in the $N$-player games. 
Moreover we need to work with a more restrictive hypotheses.
In order to formulate the result consider
\begin{itemize}
\item[\rm(i)] A filtered probability space 
$(\Omega,\mathcal{F}, \mathbf{F}=  \lbrace \mathcal{F}_t\rbrace_{t \geq 0},\mathbf{P})$ 
that satisfies the usual conditions, where all the processes are defined.
\item[\rm(ii)] Adapted independent L\'evy processes  $X,\{X^i\}_{i=1,2,\dots}$. We use the same definition for the set $\mathcal{A}$ with the clear distinction that now they are adapted to the new filtration. Nevertheless we only consider the subset of controls 
$$\mathcal{A}_{\infty}:= \lbrace \eta \in \mathcal{A}, \ X_{\infty}^{\eta} \in \mathcal{P}^{\infty} \rbrace. $$  
\item[\rm(iii)]
We denote  $\eta_i^{a,b}=(U^{i,a,b},D^{i,a,b})$ as the reflection of the process $X^i$ in the barriers $a \leq b$ (the inequality strict if the process has unbounded variation).
For simplicity and coherence  we denote its $X^{i,a,b}$ controlled process and $X^{i,a,b}_{\infty}$ its stationary adjoint random variable.
\item[\rm(iv)] There is a pair of real continuous positive functions with minimum at zero $g,h$ such that $g$ is twice contrinuously differentiable, strictly convex and. We take the cost function as
$$ c(x,y):=g(x)h(y),$$
which is clearly under \ref{A:Assumptions2MEANFIELDGAME}.
\end{itemize}
We define a \emph{vector of admissible controls} by
$$
\Lambda=(\eta_1, \dots ,\eta_N)
$$ 
and ,
$$
\aligned
\Lambda^{-i}&=(\eta_1, \dots, \eta_{i-1}, \eta_{i+1}, \dots ,\eta_N
),\\
(\eta,\Lambda^{-i})&=(\eta_1, \dots, \eta_{i-1},\eta,\eta_{i+1}, \dots ,\eta_N).
\endaligned
$$ 
for $a \leq b$, we define $\Lambda_{a,b}=((U_1^{a,b},D^{a,b}_1),\dots ,(U_N^{a,b},D^{a,b}_N))$. Moreover we denote
\begin{equation}\label{E:nashLEVY}
\bar{f}^{-i}_{\Lambda}=\frac{1}{N-1}  \sum_{j \neq i}^N f( X_{\infty}^{j,\eta_j}),
\quad
\bar{f}^{a,b,-i}=\frac{1}{N-1}  \sum_{j \neq i}^N f( X_{\infty}^{j,a,b}),
\end{equation}
and, given $\eta\in\mathcal{A}_{\infty}$,
for $(\eta, \Lambda^{-i})$, consider
\begin{align}\label{eq:NCOSTLEVY}
J_{\infty,N}^{i}(x,\eta, \Lambda^{-i})= 
\limsup_{T \rightarrow \infty} \frac{1}{T} \E_x & 
\Bigg( \int_0^T c \left(X_s^{i,\eta_i},\bar{f}^{-i}\right)ds
+q_u U^{i,\eta_i}_T+q_d D^{i,\eta_i}_T \Bigg), \nonumber \\
J_{\epsilon,N}(x,\eta,\Lambda^{-i} )  =  \E_x \Bigg( \int_{0}^{\infty}e^{-\epsilon s} & c(X_s^{i,\eta_i},\bar{f}^{-i})ds   + q_udU^{\eta_i}_s  \\ &  +q_d dD^{\eta_i}_s  +q_uu_0^{\eta_i}+q_dd_0^{\eta_i}\bigg), \nonumber
\end{align}
\begin{definition}\label{D:epsilonNashequilibriumLEVY}
For fixed $r,\epsilon>0$ and $N\in\mathbb{N}$,  a vector of admissible stationary bounded controls
$\Lambda=(\eta,\dots ,\eta_N)$  is called 
\begin{itemize}
    \item[\rm{(i)}]
an $r$-ergodic Nash equilibrium if 
for all $i$ and all $x\in\R$,
$$
J_{\infty,N}^{i}(x,\eta_i, \Lambda^{-i}) \leq J^i_{\infty,N}(x,\mu,\Lambda^{-i} )+r, \quad\text{for all } \mu \in \mathcal{A}_{\infty}. 
$$
\item[\rm{(ii)}] an $r,\epsilon$-discounted Nash equilibrium if for all $i$ and all $x \in \mathbb{R},$
$$
J_{\epsilon,N}^{i}(x,\eta_i, \Lambda^{-i}) \leq J^i_{\epsilon,N}(x,\mu,\Lambda^{-i} )+r,\quad\text{for all } \mu \in \mathcal{A}_{\infty}. 
$$
\end{itemize}

\end{definition}
First, we present a result for the class of reflecting controls, we omit the proof as it is exactly the same as \cite[Theorem 5.2]{CMO}.

 \begin{theorem}
Assume that the set of controls for each process $X^i, \ i=1, \dots ,N$, 
is the set of reflecting controls instead of $ \mathcal{A}_{\infty}$  and $f(x)=x$.
Then, if $(a,b)$ is an $\epsilon$-equilibrium (ergodic equilibrium) point for the mean-field game driven by $X$,
given $r>0$,
the vector of controls $\Lambda_{a,b}$
is an $r,\epsilon$-Nash ($r$-ergodic Nash) equilibrium for the $N$-player game,
for $N$ large enough.
 \end{theorem}
Alternatively, we can restrict the second moments to get another approximation result.
\begin{definition}
    For $K>0, \ \epsilon>0$, we define the class of controls  $\mathcal{B}^K_0, \ \mathcal{B}^k_{\epsilon} \subset \mathcal{A}_\infty$ as:

\begin{align}
  &  \mathcal{B}^K_0= \left\lbrace \eta \in \mathcal{A}_{\infty}: \ \limsup_{T \to \infty}\frac{1}{T} \E_x \left(\int_0^T g^2(X^{\eta}_s)ds \right) \leq K , \ \text{ for every } x\in \mathbb{R} \right\rbrace , \nonumber \\
& \mathcal{B}_{\epsilon}^K = \left\lbrace \eta \in \mathcal{A}_{\infty}: \ \E_x \left(\int_0^\infty g^2(X^{\eta}_s)e^{-\epsilon s} ds\right)\leq K, \ \text{ for every } x \in \mathbb{R} \right\rbrace. \nonumber
\end{align}
\end{definition}

\begin{theorem}
    Fix $N \in \mathbb{N}, \ K>0,  $ and $a < b$. Let the constant $\delta>0$ satisfy $\vert f(x)-f(y) \vert <1/N^2 $.
 \begin{itemize}
     \item[\rm{(i)}] Assume that the set of admissible controls is restricted to $\mathcal{B}^K_0$. If $(a,b)$ is an ergodic MFG-equilibrium and $\Lambda_{a,b} \in \mathcal{B}^K_0$, then $\Lambda_{a,b}$ is an $r$-ergodic Nash equilibrium for 
     $$r=\max_{z \in [a,b]}(h^2(z)) 4K e^{\frac{-2 \delta^2}{(b-a)^2N}}+2K/N .$$
     \item[\rm{(ii)}] Assume that the set of admissible controls is restricted to $\mathcal{B}^K_\epsilon$. If $(a,b)$ is an $\epsilon$-discounted MFG equilibrium and $\Lambda_{a,b} \in \mathcal{B}^K_{\epsilon}$, then $\Lambda_{a,b}$ is an $r,\epsilon$-discounted Nash equilibrium for
          $$r=\frac{\max_{z \in [a,b]}(h(z))^2 4Ke^{\frac{-2 \delta^2}{(b-a)^2N}}}{\epsilon}+\frac{2K}{N\epsilon} . $$

 \end{itemize}   
\end{theorem}
\begin{proof}
    We only prove the first claim as the other is analogue.
    First observe, as $f(x)$ is continuous, the set $f([a,b])$ is a closed interval, denote it by $[m,M]$,
and  observe that 
$$
(X^{i,a,b}_s,\bar f^{-i,a,b}_s)\in [a,b]\times[m,M], 
$$
that is a product of closed intervals. 
As for $m \leq y \leq M$, we have   
\begin{equation}\label{eq:cbounded}
    c(x,y) \leq g(x)\max_{z \in [m,M]}h(z)  =: g(x)H(m,M)  ,
\end{equation}
we define
 $$C:= \lim_{T \to \infty} 2 \left( \frac{H(m,M)}{T}\E_x \left( \int_0^T g ( X^{a,b}_s ) ds +dU_T^{a,b}q_u+dD_T^{a,b}q_d \right) \right) . $$
  Let $\mathcal{C}$ be defined as the set of strategies $\eta \in \mathcal{B}^K_0$ that satisfy 
  $$J_{\infty,N}^i (\eta, \Lambda^{-i}_{a,b})<C , \quad \text{ for all } i=1, \dots N, \ N \in \mathbb{N}. $$
 Now, observe that for all $i=1, \dots N, \ N \in \mathbb{N}$, by taking $\eta = (U^{a,b},D^{a,b})$, we have  $ J_{\infty,N}^i (\eta,\Lambda^{-i})\leq C/2< C ,$ thus
\begin{equation*}
    \inf_{\eta \in \mathcal{B}^K_0} J_{\infty,N}^i (\eta,\Lambda^{-i}_{a,b})=   \inf_{\eta \in \mathcal{C}} J_{\infty,N}^i (\eta,\Lambda^{-i}_{a,b}), \quad \text{ for all } i=1, \dots N, \ N \in \mathbb{N}.  
\end{equation*}
Furthermore, observe that

\begin{align}
J_{\infty,N}^{i}( \Lambda_{a,b})&=  J_{\infty,N}^{i}(\eta,\Lambda^{-i}_{a,b})+\big(J_{\infty,N}^{i}(\Lambda_{a,b})-J(\eta,X^{a,b}_{\infty}\big) \nonumber
\\& +\big(J(\eta,X^{a,b}_{\infty}\big)-J_{\infty,N}^{i}(\eta,\Lambda^{-i}_{a,b}) \big)\nonumber \\
                &\leq J_{\infty,N}^i(\eta,\Lambda^{-i}_{a,b})+\big(J_{\infty,N}^{i}(\Lambda_{a,b})-J((U^{a,b},D^{a,b}),X^{a,b}_{\infty}\big)\nonumber\\&+\big(J(\eta,X^{a,b}_{\infty} )-J_{\infty,N}^{i}(\eta,\Lambda^{-i}_{a,b}) \big).  \nonumber
\end{align}
Thus, to prove (\rm{i}), it is enough to demostrate 
 \begin{equation}\label{E:NaproximmationSubstraction2}
\sup_{\eta\in\mathcal{C}}\vert  J_{\infty,N}^{i}(\eta, \Lambda^{-i}_{a,b})-J(\eta, X^{a,b}_{\infty})\vert  \leq r/2,
\end{equation}
as the term $\vert J_{\infty,N}^{i} (\Lambda_{a,b})- J((U^{a,b},D^{a,b}),X^{a,b}_{\infty}) \vert$ is smaller or equal than the supremum defined in \eqref{E:NaproximmationSubstraction2}, because $(U^{a,b},D^{a,b}) \in \mathcal{C}$. For that endeavor, we need to prove that
\begin{equation}\label{eq:finalapproximation}
\left|\frac1T\E_x\int_0^T\left(c(X^{i,\eta_i}_s,\bar f^{a,b,-i})-c(X^{i,\eta_i}_s,\E_x(f(X^{a,b}_{\infty}))\right)\,ds\right|,
\end{equation}
is smaller or equal to $r/2$.
With Cauchy-Schwarz inequality and the fact that $\eta^i \in \mathcal{B}_0^K$, we deduce the expression \eqref{eq:finalapproximation} is smaller or equal than 
\begin{equation}\label{eq:finalapproximation2}
\left|\frac{K}{T}\E_x\int_0^T\left(h(\bar f^{a,b,-i})-h(\E_x(f(X^{a,b}_{\infty}))\right)^2\,ds\right|.
\end{equation}
Finally, we use Hoeffding's inequality for bounded random variables, the fact $\vert h(x)-h(y)\vert \leq b^2+a^2$ and the definition of $\delta$ to conclude \eqref{eq:finalapproximation2} is smaller or equal than $r/2$, thus concluding the proof of the theorem.

\end{proof}

%

%

\end{document}